\documentclass{article}
\usepackage{fullpage}

\usepackage{amsfonts}
\usepackage{amssymb}
\usepackage{amsmath}

\usepackage{hyperref}
\newcommand{\arxiv}[1]{\href{http://www.arXiv.org/abs/#1}{arXiv:#1}}
\usepackage{tikz}
\usetikzlibrary{arrows}
\newcommand {\junk}[1]{}

\usepackage{amsthm}

\newtheorem{maintheorem}{Main Theorem}
\newtheorem{theorem}{Theorem}[section]
\newtheorem{lemma}[theorem]{Lemma}
\newtheorem{proposition}[theorem]{Proposition}
\newtheorem{corollary}[theorem]{Corollary}
\theoremstyle{definition}
\newtheorem{definition}[theorem]{Definition}
\newtheorem{remark}[theorem]{Remark}

\numberwithin{equation}{section}

\def\crit{{\mathcal G}^c}

\def\R{{\mathbb R}}
\def\bool{{\mathbb B}}

\def\Rp{\R_+}
\def\Rpn{\Rp^n}
\def\Rpnn{\Rp^{n\times n}}
\def\cycle{C}

\def\digr{{\mathcal G}}

\def\ind{T}

\def\wiel{\operatorname{Wi}}

\title{Generalizations of Bounds on the Index of Convergence to Weighted
Digraphs}

\author{
Glenn Merlet\footnote{Universit\'e d'Aix-Marseille, CNRS, IML, 13009
MARSEILLE, France. Email: {\tt glenn.merlet@gmail.com}}
\and
Thomas Nowak\footnote{Laboratoire d'Informatique, \'Ecole polytechnique,
France. Email: {\tt nowak@lix.polytechnique.fr}}
\and
Hans Schneider\footnote{University of Wisconsin-Madison, USA. Email: {\tt
hans@math.wisc.edu}. Partially supported by EPSRC grant RRAH15735.}
\and
Serge\u{\i} Sergeev\footnote{University of Birmingham, School of Mathematics,
UK. Email: {\tt sergiej@gmail.com}. Partially supported by EPSRC grant
RRAH15735, RFBR grant 12-01-00886, and RFBR-CNRS grant 11-0193106.}
}

\date{}

\begin{document}

\maketitle

\begin{abstract}
We study sequences of optimal walks of a growing length, in
weighted digraphs, or equivalently, sequences of
entries of max-algebraic matrix powers with growing exponents. It
is known that these sequences are eventually periodic when the digraphs are strongly connected.
The transient of such periodicity depends, in
general, both on the size of digraph and on the magnitude of the
weights. In this paper, we show that some bounds on the indices of
periodicity of (unweighted) digraphs, such as the bounds of
Wielandt, Dulmage-Mendelsohn, Schwarz, Kim and Gregory-Kirkland-Pullman, apply to the weights
of optimal walks when one of their ends is a critical node.

\medskip


\medskip

\end{abstract}

\section{Introduction}

We show that six known bounds for the index of convergence (transient of periodicity) of an unweighted  digraph also apply to weighted digraphs, namely, to transients of rows and columns
with critical indices.

The origin of the first of these known bounds lies in Wielandt's
well-known paper \cite{Wie-50} where an upper bound for the transient
of a primitive nonnegative matrix was asserted without
proof\footnote{Wielandt's proof was published later in
\cite{Sch-02}.}. Dulmage and Mendelsohn \cite{DM-62}  provided a
proof of this result by interpreting it in terms of digraphs and
they sharpened the result by using as additional information in the
hypotheses the length of the smallest cycle of the
digraph\footnote{Denardo~\cite{Den-77} later rediscovered their
result.}.
 Schwarz \cite{Sch-70} generalized Wielandt's
  result to apply to all strongly connected digraphs  by using Wielandt's bound for the cyclicity
classes of the digraph, see also Shao and Li~\cite{SL-87}.
Kim's \cite{Kim-79} bound encompasses the first three and can be proved using Dulmage
and Mendelsohn's bound in the cyclicity classes.

We also generalize another bound by Kim~\cite{Kim-79}, and a bound
by Gregory-Kirkland-Pullman~\cite{GKP-95} which depend on Boolean
rank.

The six bounds mentioned above are stated in Theorem~\ref{t:allbounds} and Theorem~\ref{t:rank:bool}
after the requisite definitions.
Our generalizations to weighted digraphs are stated in Main Theorem~\ref{t:mainres}
and Main Theorem~\ref{t:rank}, and subsequently
proved in Sections~\ref{s:DM}--\ref{s:rank}.


 We exploit the natural connection between weighted digraphs and nonnegative
matrices in the max (-times) algebra
 just as the bounds that we take for our starting points connect unweighted digraphs and Boolean matrices.

\section{Preliminaries and Statement of Results}

\subsection{Digraphs, Walks, and Transients}
\label{ss:digraphs}
Let us start with some definitions.
\begin{definition}[Digraphs]
\label{d:digr}
A {\em digraph} is a pair $\digr=(N,E)$ where
$N$ is a set of nodes and $E\subseteq N\times N$ is
a set of edges. 
\end{definition}
\begin{definition}[Walks and cycles]
\label{d:walk}
A {\em walk} in $\digr$ is a finite sequence $W=(i_0,i_1,\ldots,i_t)$ of
nodes such that
each pair $(i_0,i_1),(i_1,i_2),\ldots,(i_{t-1},i_t)$ is an edge of $\digr$ (that is, belongs to $E$).
Here, the nodes $i_0$, resp.~$i_t$ are the {\em start} resp.\ the {\em end} nodes of
the walk. 

The number $t$ is
the {\em length} of the walk, and we denote it by~$\ell(W)$. 

When $i_0=i_t$, the walk is
{\em closed}. If, in a closed walk, none of the
nodes except for the start and the end appear more than once, the walk is called a {\em cycle}.
If no node appears more than once, then the walk is called a {\em path}.
A walk is {\em empty\/} if its length is~$0$.
\end{definition}

To a digraph $\digr=(N,E)$ with $N=\{1,\ldots,n\}$, we can associate a
Boolean matrix $A=(a_{i,j})\in\bool^{n\times n}$ defined by
\begin{equation}
\label{e:sij}
a_{i,j}=
\begin{cases}
0 &\text{if $(i,j)\notin E$}\\
1 &\text{if $(i,j)\in E$}
\enspace.
\end{cases}
\end{equation}
Conversely, one can associate a digraph to every square Boolean matrix.
The connectivity in $\digr$ is closely related to the Boolean matrix powers of
$A$. By the {\em Boolean algebra\/} we mean the set $\bool=\{0,1\}$ equipped with the logical operations
of conjunction $a\cdot b$ and disjunction $a\oplus b=\max(a,b)$,
for $a,b\in\bool$. The Boolean multiplication of two matrices
$A\in\bool^{m\times n}$ and
$B\in\bool^{n\times q}$
is defined by $(A\otimes B)_{i,j}:=\bigoplus_{k=1}^n (a_{i,k}\cdot b_{k,j})$, and then we also have
Boolean matrix powers $A^{\otimes t}:=\underbrace{A\otimes\ldots\otimes
A}_{t\text{ times}}$.
The $(i,j)$\textsuperscript{th} entry of
$A^{\otimes t}$ is denoted by~$a_{i,j}^{(t)}$.

The relation between Boolean powers of~$A$ and connectivity in~$\digr$ is
based on the following fact: $a_{i,j}^{(t)}=1$ if and only if $\digr$ contains
a walk of length $t$
from $i$ to $j$.

Let $\digr$ be a digraph with associated matrix~$A\in\bool^{n\times n}$.
The sequence of Boolean matrix powers
$A^{\otimes t}$ is eventually periodic, that is, there exists a positive
$p$
such that
\begin{equation}\label{e:periods}
A^{\otimes (t+p)}=A^{\otimes t}
\end{equation}
for all $t$ large enough. 

\begin{definition}[Eventual period]
\label{d:period}
Each $p$ satisfying~\eqref{e:periods} is called an {\em eventual
period\/} of~$A$. 
\end{definition}

The set of nonnegative~$t$ satisfying~\eqref{e:periods} is
the same for all eventual periods~$p$. 

\begin{definition}[Transient of digraphs]
The least nonnegative number $t$ satisfying~\eqref{e:periods} for some 
(and hence for all) $p$ is called the {\em transient (of periodicity) 
of $\digr$}; we denote it by
$\ind(\digr)$. 
\end{definition}

See~\cite{BR} for general introduction to the theory
of digraphs and~\cite{LS-93} for a survey on their
transients. 
In the literature, $\ind(\digr)$ is often
called the {\em index of convergence}, or the {\em exponent\/} of $\digr$. 

\begin{definition}[Powers of digraphs]
The digraph associated with $A^{\otimes t}$ is denoted by $\digr^t$. Such graphs
will be further referred to as the {\em powers} of $\digr$.
\end{definition}

\begin{definition}[Cyclicity and primitivity]
For a strongly connected digraph~$\digr$, its {\em
cyclicity\/} is the greatest common divisor of the lengths of all cycles of~$\digr$. If $d=1$, then~$\digr$ is called
{\em primitive}, otherwise it is called {\em imprimitive}.
\end{definition}

The cyclicity $d$ of~$\digr$ can be equivalently defined as
the least eventual period~$p$ in~\eqref{e:periods}. 
\if{
If~$\digr$ is primitive, i.e. if $d=1$, the transient of $\digr$
convergence is
also called the {\em exponent\/} and it is equal to the
smallest nonnegative~$t$ such that all entries of $A^{\otimes t}$ are equal
to~$1$.
}\fi
If~$\digr$ is strongly connected, then its cyclicity is the smallest eventual
period of its sequence of powers.
Let us recall the following basic observation from~\cite{BR}.

\begin{theorem}[{\cite[Theorem 3.4.5]{BR}}]
\label{t:graphpowers} Let $\digr$ be a strongly connected graph with
cyclicity $d$. For each $k\geq 1$, graph $\digr^k$ consists of
$\gcd(k,d)$ isolated strongly connected components, and every
component has cyclicity~$d/\gcd(k,d)$.
\end{theorem}

We have an important special case when $k=d$.
\begin{definition}[Cyclicity classes]
\label{d:cyclasses}
According to Theorem~\ref{t:graphpowers}, $\digr^d$ has exactly~$d$ strongly
connected components, each of cyclicity~$1$.
The node sets of these components are called the {\em cyclicity classes\/}
of~$\digr$.
\end{definition}

In terms of walks, nodes $i$ and $j$ belong to the same cyclicity class if and only if there is a
walk from $i$ to $j$ whose length is a multiple of $d$. More generally, for each $i$ and $j$ there is
an integer~$s$  with $0\leq s\leq d-1$ such that the length of every walk connecting $i$ to $j$ is congruent
to $s$ modulo $d$. This observation defines the {\em circuit of cyclicity classes}, being
crucial for the description of $\digr^ t$ in the periodic regime.

Let us introduce the following 
two definitions. The first of them is given
in honor of the first paper on the subject by Wielandt~\cite{Wie-50}.

\begin{definition}[Wielandt number]
By {\em Wielandt number} we mean the following function:
\begin{equation}
\label{wienumber}
\wiel(k):=
\begin{cases}
0 & \text{if $k=1$,}\\
(k-1)^2+1 & \text{if $k>1$.}
\end{cases}
\end{equation}
\end{definition}

\begin{definition}[Girth]
\label{d:girth}
The {\em girth\/} of $\digr$, denoted it by $g(\digr)$,
is the smallest length of a nonempty cycle in $\digr$.
\end{definition}

We will be interested in the following bounds on $\ind(\digr)$. 

\begin{theorem}
\label{t:allbounds}
Let $\digr$ be a strongly connected digraph with~$n$ nodes, cyclicity~$d$, and
girth~$g$.
The following upper bounds on the transient of~$\digr$ hold:
\begin{itemize}
\item[{\rm (i)}] (Wielandt~\cite{Wie-50,Sch-02}) If $d=1$, then
$\ind(\digr)\leq\wiel(n)$;
\item[{\rm (ii)}] (Dulmage-Mendelsohn~\cite{DM-62})
If $d=1$, then $\ind(\digr)\leq (n-2)\cdot g+n$;
\item[{\rm (iii)}] (Schwarz~\cite{Sch-70,SL-87}) $\displaystyle\ind(\digr)\leq
d\cdot\wiel\left(\left\lfloor\frac{n}{d}\right\rfloor\right)
+(n \bmod d)$;
\item[{\rm (iv)}]  (Kim~\cite{Kim-79}) $\displaystyle\ind(\digr)\leq
\left(\left\lfloor\frac{n}{d}\right\rfloor-2\right)\cdot g + n$.
\end{itemize}
\end{theorem}


\begin{remark}
\label{r:allbounds} {\rm Clearly, the bound of Schwarz is tighter
than the bound of Wielandt in the imprimitive case ($d>1$), while in
the primitive case ($d=1$) they are equal to each other. The bound
of Kim is in the same relation with the bound of Dulmage and
Mendelsohn. Further, the bound of Dulmage and Mendelsohn is tighter
than that of Wielandt when $g<n-1$ and both bounds are equal when
$g=n-1$. In the remaining case $g=n$, the graph consists of a single
Hamiltonian cycle and periodicity starts from the very beginning, i.e.,
$T(\digr)=1$.

Let us show that the bound of Schwarz can be deduced from the bound
of Kim. Firstly, it can be seen that by substituting
$g=d(\left\lfloor\frac{n}{d}\right\rfloor-1)$ in Kim's bound and
using the identity $n=d\left\lfloor\frac{n}{d}\right\rfloor+(n \bmod
d)$ we obtain the bound of Schwarz. Hence the bound of Kim is
tighter when $\frac{g}{d}<\left\lfloor\frac{n}{d}\right\rfloor-1$,
and the bounds are equal when
$\frac{g}{d}=\left\lfloor\frac{n}{d}\right\rfloor-1$.

Consider the remaining case
$\frac{g}{d}=\left\lfloor\frac{n}{d}\right\rfloor$. By definitions of~$g$ and~$d$, the
length of any cycle on~$\digr$ equals to $g+td$ for some $t\geq 0$. Since
$\frac{g}{d}=\left\lfloor\frac{n}{d}\right\rfloor$, all 
cycles are of length~$g$, thus $g=d$ and
$\left\lfloor\frac{n}{d}\right\rfloor=1$. Therefore, both bounds
equal to $(n \bmod d)=n-g$.
}
\end{remark}

We are also interested in the improvements of
Theorem~\ref{t:allbounds} in terms of the  {\em factor rank\/} of a
matrix~$A\in\bool^{n\times n}$ (also known as the Boolean rank or
Schein rank). 

\begin{definition}[Factor rank in Boolean algebra]
The {\em Factor rank\/} of $A$ is
the least number~$r$ such that
\begin{equation}\label{e:rank}
A = \bigoplus_{\alpha=1}^r v_\alpha \otimes w^T_\alpha
\end{equation}
with Boolean vectors $v_1,w_1,\dots,v_r,w_r\in\bool^n$.
\end{definition}


The factor rank of~$A$ is at most~$n$ since~\eqref{e:rank} holds when choosing
$r=n$ and the $w_\alpha$ to be the unit vectors and the $v_\alpha$ to be the
columns of~$A$.

The following bounds involving the factor rank were
established:

\begin{theorem}\label{t:rank:bool}
Let~$\digr$ be a strongly connected primitive digraph with girth $g$, and let the associated
matrix of $\digr$ have factor rank~$r$.
The following upper bounds on the transient of~$\digr$ hold:
\begin{itemize}
\item[{\rm (i)}] (Gregory-Kirkland-Pullman~\cite{GKP-95})
$\ind(\digr)\leq\wiel(r)+1$;
\item[{\rm (ii)}] (Kim~\cite{Kim-79})
$\ind(\digr)\leq (r-2)\cdot g+r+1$.
\end{itemize}
\end{theorem}


In fact, we will show that the bounds in Theorem~\ref{t:rank:bool} also hold
for
non-primitive matrices and that the analogous stronger bounds of Schwarz and
Kim
with the factor rank instead of~$n$ are true. See Main Theorem~\ref{t:rank}.




\if{
\begin{theorem}[{\cite[Theorem 3.4.5]{BR}}]
\label{t:graphpowers}
Let $\digr$ be a strongly connected graph with cyclicity $d$.
\begin{itemize}
\item[{\rm (i)}] For each $k\geq 1$, graph $\digr^k$ consists of
gcd$(k,d)$ isolated strongly connected components.
\item[{\rm (ii)}] Each component of $\digr^k$ has cyclicity $d/$gcd$(k,d)$ and
its node set consists of $d/$gcd$(k,d)$ cyclicity classes of $\digr$.
\end{itemize}
\end{theorem}

For max algebra,  it is interesting to analyze the situation when a graph
$\digr$ of cyclicity $d$ has a strongly connected subgraph $\digr'$ of cyclicity $\sigma$.
Note that in this case, $\sigma$ is always a multiple of $d$.

\begin{corollary}
\label{c:grsubgr}
Let $\digr$ be a strongly connected graph of cyclicity $d$, and
let $\digr'$ be a strongly connected subgraph of $\digr$, with node set $N'$ and cyclicity $\sigma$.
Let $k\geq 1$.
\begin{itemize}
\item[{\rm (i)}] The node set of any closed walk of $\digr$
 contains a node from each component of $\digr^k$. In particular, this also holds for $N'$.
 \item[{\rm (ii)}] There are gcd$(k,\sigma)$ components of $(\digr')^k$
such that the union of their node sets is $N'$. Each component of $\digr^k$
contains at least one component of $\digr'^k$.
\end{itemize}
\end{corollary}
\begin{remark}
\label{r:grsubgr}
{\rm 
It can be shown that each component of $\digr^k$
contains exactly gcd$(k,\sigma)/$gcd$(k,d)$ components of $(\digr')^k$,
but we do not need this precision here.}
\end{remark}
\begin{proof}
(i): Take a closed walk $W$ of $\digr$ and recall that
it contains a node from each
cyclicity class of $\digr$. By Theorem~\ref{t:graphpowers}, the node set of each component
of $\digr^ k$ consists of several cyclicity classes, hence the
node setoff $W$ contains a node from each component of $\digr^k$.
For $N'$, take any closed walk in $\digr'$ that passes through every node of $\digr'$.

(ii): The first part of the claim follows from Theorem~\ref{t:graphpowers} part (i)
applied to $\digr'$. Next, part (i) implies that for each component of $\digr^k$
there is a component $H$ of $\digr'^k$ containing a node from that component of $\digr^k$. But
since the components of $\digr^k$ are isolated from each other, $H$ is contained in $\digr^k$.
\end{proof}

}\fi

\subsection{Weighted Digraphs and Max Algebra}


In a {\em weighted digraph\/} $\digr$, every edge $(i,j)\in E$ is
weighted by some {\em weight\/} $a_{i,j}$.
We consider the case of nonnegative
weights $a_{i,j}\in\Rp$ and 

\begin{definition}[Weight of a walk]
\label{d:weight}
The {\em weight} of a walk $W=(i_0,i_1,\ldots,i_t)$
is the product
\begin{equation}\label{e:def:weight}
p(W):=a_{i_0,i_1} \cdot a_{i_1,i_2}\cdots a_{i_{t-1},i_t}
\enspace.
\end{equation}
\end{definition}
Recall that the length of a walk $W=(i_0,i_1,\ldots,i_t)$ is 
$\ell(W)=t$.

Another common definition is letting edge weights be arbitrary reals and the
weight of walks be the sum of the weights of its edges.
One can navigate between these two definitions by taking the logarithm and the
exponential.

By {\em max algebra} we understand the set of nonnegative real numbers~$\Rp$
equipped with the usual multiplication $ab:=a\cdot b$ and tropical
addition $a\oplus b:=\max(a,b)$. This arithmetic is extended to
matrices and vectors in the usual way, which leads to max-linear
algebra, i.e. the theory of max-linear systems~\cite{ABG,But:10}.
The product of two matrices $A\in\Rp^{m\times n}$ and
$B\in\Rp^{n\times q}$ is defined by $(A\otimes
B)_{i,j}:=\max_{1\leq k\leq n} a_{i,k}\cdot b_{k,j}$, which defines
the max-algebraic matrix powers
$A^{\otimes t}:=\underbrace{A\otimes\ldots\otimes A}_{t \text{ times}}$.
The $(i,j)$\textsuperscript{th} entry of
$A^{\otimes t}$ will be denoted by $a_{i,j}^{(t)}$.
Boolean matrices are a special
case of max-algebraic matrices.

The walks of maximum weight in $\digr$ are
closely related with the entries of max-algebraic powers of the associated
nonnegative matrix of weights $A=(a_{i,j})$.
Conversely, one can associate a weighted digraph~$\digr(A)$ to every square
max-algebraic matrix~$A$.
The connection between max-algebraic powers and weights of walks is
based on the following fact called the {\em
optimal walk interpretation\/} of max-algebraic matrix powers:
$a_{i,j}^{(t)}$ is the maximum weight of all walks of length~$t$
from~$i$ to~$j$, or~$0$ if no such walk exists.

\begin{definition}[Maximum cycle geometric mean]
\label{d:mcgm}
If $\digr(A)$ has at least one 
nonempty cycle, then 
the maximum geometric cycle mean of $A\in\Rpnn$ is equal to
\begin{equation}
\label{e:mcgm}
\begin{split}
\lambda(A)&:=
\max \big\{ p(\cycle)^{1/\ell(\cycle)} \mid C\text{ is a nonempty cycle in
}\digr(A) \big\}
\end{split}
\end{equation}
Set $\lambda(A)=0$ if no nonempty cycle in~$\digr(A)$ exists.
\end{definition}

\begin{definition}[Critical graph]
\label{d:crit}
Let $A\in\Rpnn$. The cycles of $\digr(A)$ at which the maximum geometric cycle mean is attained are
called {\em critical}, and so are all nodes and edges that belong to
them. The set of all critical nodes is denoted by $N_c(A)$ and the set of all critical edges
is denoted by $E_c(A)$. 
The {\em critical graph}, denoted by $\crit(A)$, consists of
all critical nodes and all critical edges.
\end{definition}


As we have $\lambda(\alpha \cdot A)=\alpha\cdot \lambda(A)$ for all
$\alpha\in\Rp$ and $A\in\Rpnn$, we also have
$\lambda\big(A/\lambda(A)\big)=1$ whenever $\lambda(A)\neq0$. It is
$\lambda(A)\neq 0$ if and only if $\digr(A)$ contains a nonempty
cycle. In this case, the equality $A^{\otimes t} = \lambda(A)^t
\cdot \big(A/\lambda(A)\big)^{\otimes t}$ implies that we can indeed
assume $\lambda(A)=1$ without loss of generality when studying the
sequence of max-algebraic matrix powers. We will indeed assume
$\lambda(A)=1$ in the rest of the paper. It means that we avoid the
case when $\lambda(A)=0$. This case is trivial because there are no
critical nodes. Moreover, since there are no closed walks
on~$\digr(A)$, there are no walks with length more than~$n-1$, so
$A^{\otimes n}=0$.

The following definition is standard.

\begin{definition}[Irreducibility]
\label{d:irred}
A matrix $A\in\Rpnn$ is called {\em irreducible} if $\digr(A)$ is strongly connected, i.e., if for 
each $i,j\in N$ there exists a walk of nonzero weight whose starting node is $i$ and
end node is $j$. 

$A\in\Rpnn$ is called {\em reducible} if it is not irreducible.
\end{definition}

Cohen et al.~\cite{CDQV-83} first proved that
the sequence of max-algebraic matrix powers
of an irreducible
matrix~$A$ with~$\lambda(A)=1$
is eventually periodic.
\begin{definition}[Transient of matrices]
\label{d:transient}
The least nonnegative~$t$ satisfying $A^{\otimes(p+t)}=\lambda(A)^p\cdot A^{\otimes t}$ for some $p>0$
is called the {\em
transient\/} of~$A$ and denoted by $T(A)$.
\end{definition}

The transient of~$A$ depends not only on the nodes and edges in~$\digr(A)$, but
also on the specific weights in $A$.
It was studied by several authors,
including
Hartmann and Arguelles~\cite{HA-99},
Bouillard and Gaujal~\cite{BG-01},
Soto y Koelemeijer~\cite{SyK-03},
Akian et al.~\cite[Section 7]{AGW},
and
Charron-Bost et al.~\cite{CBFN-12}.
However, none of the upper bounds on the transient were generalizations of
any of the Boolean bounds of Theorem~\ref{t:allbounds} (in the sense that the
bounds of that theorem would be immediately recovered when specializing these results
to Boolean matrices). To see that $T(A)$ depends also on the weights of $A$, consider
\begin{equation*}
A=
\begin{pmatrix}
1 &\epsilon\\
\epsilon & 1/e
\end{pmatrix},
\end{equation*}
and observe that the transient of $(a_{2,2}^{(t)})_{t\geq 1}$ is equal to
$\lceil -2\log\epsilon\rceil$ if $\epsilon<1$.

In the present paper, we
generalize all the bounds in Theorem~\ref{t:allbounds} to the weighted
case.
We do this not by giving bounds on the transient of~$A$, but by giving bounds
on the transients of the critical rows and columns of~$A$. Let us give the corresponding definition.

\begin{definition}[Row and column transients]
\label{d:rowcolumn}
Let $A\in\Rpnn$.
The least nonnegative $t$ satisfying 
$a^{(p+t)}_{k,i}=\lambda(A)^p\cdot a^{(t)}_{k,i}$ for a fixed $k\in N$, all $i\in N$ and some $p>0$
is called the transient of the $k$th row of $A$  and 
denoted by $T_k(A)$, or just $T_k$ if the matrix is clear from the context.

Similarly, the least nonnegative $t$ satisfying 
$a^{(p+t)}_{i,k}=\lambda(A)^p\cdot a^{(t)}_{i,k}$ for a fixed $k\in N$, all $i\in N$ and some $p>0$
is called the transient of the $k$th column of $A$.
\end{definition}

\begin{remark}
\label{r:makesense}
{\rm  This definition makes sense also when $A$ is reducible. However, in that case the
transient is defined only for some rows and columns. Below we are interested only in the 
case of critical rows and columns, for which the transients always exist regardless of 
the irreducibility of $A$.
}
\end{remark}

\begin{remark}
\label{r:boolmax}
{\rm In the Boolean case, all rows and columns are critical, hence we are really
generalizing the Boolean bounds.
We are
motivated by a result of Nachtigall~\cite{Nacht}
who showed that the
transient of critical rows and columns does not exceed~$n^2$. Later, Sergeev
and Schneider~\cite{SerSch} conjectured that, in fact, this
transient should not exceed $\wiel(n)$.
In particular, we prove this conjecture.}
\end{remark}

The following is the first main result of the paper.
\begin{maintheorem}
\label{t:mainres}
Let $A\in\Rpnn$ be irreducible and let~$k$ be a critical node.
Denote by $d$ the cyclicity of $\digr(A)$,
by $H$ the component of the critical graph $\crit(A)$ containing~$k$, and by $\lvert H\rvert$ the number
of nodes in $H$.
The following quantities are upper bounds on the transient of the $k$\textsuperscript{th} row and the $k$\textsuperscript{th}
column:
\begin{itemize}
\item[{\rm (i)}] (Wielandt bound) $\wiel(n)$
\item[{\rm (ii)}] (Dulmage-Mendelsohn bound)
$ (n-2)\cdot g(H) + \lvert H\rvert$
\item[{\rm (iii)}] (Schwarz bound) $\displaystyle
d\cdot\wiel\left(\left\lfloor\frac{n}{d}\right\rfloor\right)
+ (n\bmod d)$
\item[{\rm (iv)}]  (Kim bound) $\displaystyle
\left(\left\lfloor\frac{n}{d}\right\rfloor-2\right)\cdot g(H)+\min(n, \lvert
H\rvert+ \left(n\bmod d)\right)$
\end{itemize}
The first two bounds also hold in the case when $A$ is reducible.
\end{maintheorem}

\begin{remark}
\label{r:genperiod}
{\rm Denote by $\gamma(\crit(A))$ the least common multiple of the cyclicities of
all strongly connected components of $\crit(A)$.  This number is well-known to be the least eventual period
of the sequence $(A^{\otimes t})_{t\geq 1}$ when $A$ is irreducible (see~\cite{CDQV-83,But:10}). It is also
the least eventual period of the sequence of submatrices of $A^{\otimes t}$ extracted from the critical
rows or the critical columns (also in the reducible case). For an individual critical row or column, the least eventual period can be shown to
be equal to the cyclicity of the component of $\crit(A)$ where the index of
that row or column lies (see Remark~\ref{r:remHperiod}).
}
\end{remark}

\begin{remark}
\label{r:mainres} {\rm As in the Boolean case, the bound of Schwarz (resp.\ Kim) is tighter than
the bound of Wielandt (resp.\ Dulmage and Mendelsohn) when the corresponding component of $\digr$ is imprimitive.
Moreover, the bound of Wielandt is never tighter than that of Dulmage and Mendelsohn when $g(H)\leq n-1$. Unlike for the unweighted graphs,
the case $g(H)=n$ is non-trivial and will be treated below. 
Likewise, the bound of Schwarz is never tighter than the bound of Kim when
$\frac{g(H)}{d}\leq \left\lfloor\frac{n}{d}\right\rfloor-1$, but the case $\frac{g(H)}{d}=\left\lfloor\frac{n}{d}\right\rfloor$
has to be treated separately. Here we prefer to deduce the bound of Kim from the bound of Dulmage and Mendelsohn  in the same way as the bound of Schwarz is derived from the bound of Wielandt
(similar to the approach of Shao and Li~\cite{SL-87}).
}
\end{remark}

\begin{definition}[Factor rank in max algebra]
In max algebra, {\rm factor rank\/} of $A\in\Rpnn$ is the least number $r$
such that~\eqref{e:rank} holds for some vectors
$v_1,w_1,\ldots,v_r,w_r\in\Rpn$. 
\end{definition}

In our next main result, we show
that the results of Main Theorem~\ref{t:mainres} can be improved by
means of factor rank, thus obtaining a max-algebraic extension of
Theorem~\ref{t:rank:bool}.

\begin{maintheorem}
\label{t:rank} Let $A\in\Rpnn$ be irreducible. Denote by $d$ the
cyclicity of $\digr(A)$ and by~$r$ the factor rank of~$A$. Let~$k$ be
critical. Denote by $H$ the component of the critical
graph $\crit(A)$ containing~$k$, and let $h\leq\min(\lvert H\rvert,r)$ be the
parameter defined below in~\eqref{e:hdef}. The following upper
bounds on the transient of the $k$\textsuperscript{th} row and $k$\textsuperscript{th} column hold:
\begin{itemize}
\item[{\rm (i)}] $\displaystyle
\wiel\left(r\right) + 1$;
\item[{\rm (ii)}] $\displaystyle
\left(r-2\right)\cdot g(H)+ h + 1$.

\item[{\rm (iii)}] $\displaystyle
d\cdot\wiel\left(\left\lfloor\frac{r}{d}\right\rfloor\right)
+ (r\bmod d) + 1$;
\item[{\rm (iv)}] $\displaystyle
\left(\left\lfloor\frac{r}{d}\right\rfloor-2\right)\cdot g(H)+ \min(r,h+(r\bmod d)) + 1$.
\end{itemize}
The first two bounds apply to reducible matrices as well.
\end{maintheorem}

\begin{remark}
{\rm
While all parameters appearing in the bounds of Main Theorem~\ref{t:mainres} only depend
on the unweighted digraphs underlying~$\digr(A)$ and~$\crit(A)$,
the factor rank~$r$ of Main Theorem~\ref{t:rank} depends on the values of~$A$,
i.e. on the weights on~$\digr(A)$.
}
\end{remark}


The next five sections of the paper contain the proofs of
Main Theorem~\ref{t:mainres} and Main Theorem~\ref{t:rank}.
That is, we will prove that $T_k(A)$ for a critical index $k$
is less than any of the quantities in Main~Theorem~\ref{t:mainres}
and Main~Theorem~\ref{t:rank}.
Applying the result to the transposed matrix~$A^T$,
we see that the bounds also hold for the transients of the columns.

The proofs do not use the results of Theorem~\ref{t:allbounds} or Theorem~\ref{t:rank:bool} and hence, in
particular, we give new proofs for those classical results.

\if{
Note that, in contrast to the non-weighted case, only~(iii) encompasses~(i) in
Theorem~\ref{t:mainres};
all other pairs of bounds are mutually incomparable.
This is due to the occurrence of the parameters $g(H)$ and $\lvert H\rvert$ of
the critical graph, which do not have an equally strong connection to~$d$
and~$n$ as in the Boolean case.
}\fi


\subsection{Visualization}
\label{ss:vis}

In the end of this section, let us recall a result on diagonal matrix
scaling, which we will use.

Let $X$ be an $n\times n$ nonnegative diagonal matrix, i.e. a
matrix whose diagonal entries are positive and whose off-diagonal
entries are zero. Consider the operation of {\em diagonal similarity
scaling\/} $A\mapsto X^{-1}AX$, applied to any $A\in\Rpnn$. It can be
checked that the diagonal similarity scaling preserves $\lambda(A)$
and commutes with max-algebraic matrix powering: for $B=X^{-1}AX$ we
have $\lambda(A)=\lambda(B)$ and $B^{\otimes t}=X^{-1}A^{\otimes t}X$. Hence, to analyze
max-algebraic matrix powers we will use a particular
``canonical'' form that can be always reached by means of a diagonal
similarity scaling.

\begin{definition}[Visualization]
\label{d:visualize}
A matrix $A\in\Rpnn$ is called {\em visualized\/} if it has
$a_{i,j}\leq\lambda(A)$ for all $i,j\in N$. For a visualized matrix, it
also follows that $a_{i,j}=\lambda(A)$ for all critical edges $(i,j)\in E_c$.

Further if $a_{i,j}=\lambda(A)$ holds only for all critical edges $(i,j)\in E_c$, matrix
$A$ is called {\em strictly visualized.\/}
\end{definition}

\begin{remark}\label{r:visualize}
It is known that every nonnegative matrix with positive maximum geometric cycle mean can be brought to a visualized
form by means of a diagonal similarity scaling. Moreover, every nonnegative matrix can be brought
to a strictly visualized form~\cite{SSB}. Hence in our analysis
of max-algebraic powers, we can assume without loss of generality that $A$ is
visualized (and, moreover, strictly visualized), which we do in the rest of the paper.
Since we also assume~$\lambda(A)=1$, it means that all entries are between~$0$ and~$1$ and
critical edges are exactly edges with weight~$1$.
\end{remark}

An early use of visualization scaling (unrelated to max algebra) can
be found in Fiedler and Pt\'{a}k~\cite{FP}, and the scaling was
studied in more detail in~\cite{SSB}. For a short survey on the use of
visualization scaling
in max algebra see~\cite{Ser-13}. Let us conclude with the following
observation concerning
the visualization of max-algebraic powers.
Recall that $\digr^t$ denotes the $t$\textsuperscript{th} power of an arbitrary digraph~$\digr$.

\begin{lemma}[cf.~{\cite[Lemma 2.9]{BSST},\cite{Ser-09}}]
\label{l:CAk}
Let $A\in\Rpnn$ and $t\geq 1$.
\begin{itemize}
 \item[{\rm (i)}] $\crit(A^{\otimes t})=(\crit(A))^t$,
\item[{\rm (ii)}] If $A$ is visualized (or strictly visualized), then so is $A^{\otimes t}$.
\end{itemize}
\end{lemma}

\if{
\begin{theorem}
\label{circulant} Let $A\in\Rp^{n\times n}$ be a visualized matrix,
and let $i,j\in N_c(A)$ be such that $[i]\to_l [j]$ for some $l\geq
0$. Then for all $r\geq T_c(A)$
\begin{equation}
\label{period} A_{\cdot i}^r=A_{\cdot j}^{r+l},\
A_{j\cdot}^r=A_{i\cdot}^{r+l}.
\end{equation}
\end{theorem}
}\fi


\section{Proof of Dulmage-Mendelsohn Bound}
\label{s:DM}

In this section, we want to prove the second bound of Main Theorem~\ref{t:mainres}.
In fact, we first argue that $T_k(A)\le n-1$ when~$a_{k,k}=\lambda(A)$, from which we get
$T_k(A)\le g(H)(n-1)$ for $H$ a component of the critical graph and $k$ in a minimal cycle of~$H$,
and finally the general bound.

The proof splits into the following lemmas, all of them apply to any $A\in\Rpnn$ and will be used again later. 
Without loss of generality, we will assume in the proofs that~$\lambda(A)=1$ and $A$~is visualized, that is all edges have weight at most~$1$ and all critical edges have weight~$1$. (See remark~\ref{r:visualize})

The first lemma gives a means to bound the transient of a row of the matrix.
More specifically, it states that a row is periodic as soon as its entries are
equal to that of a higher matrix power.
The proof is a straightforward calculation.

\begin{lemma}\label{lem:trans:multiples}
Let~$k\in N$. Further assume that there exist $r<s$ such
that $a^{(r)}_{k,j} = a^{(s)}_{k,j}$ for all $j\in N$. Then $T_k\leq
r$.

In particular, $T_k(A)\leq m\cdot T_k\big(A^{\otimes m}\big)$ for all $m\geq1$.
\end{lemma}
\begin{proof}
Set $p=s-r$ and let $j\in N$. Then, for all $t\geq r$:
\[
a^{(t+p)}_{k,j} = \max_{l\in N} a^{(s)}_{k,l} \cdot a^{(t+p-s)}_{l,j}
= \max_{l\in N} a^{(r)}_{k,l} \cdot a^{(t-r)}_{l,j} = a^{(t)}_{k,j}
\]
This concludes the proof.
\end{proof}

The next lemma shows that a stronger form of the Weighted
Dulmage-Mendelsohn bound holds if~$k$ lies on a critical cycle of
length~$g(H)$. 
Its proof uses the fact that~$k$ lies on a critical cycle of length~$1$ in the
digraph of the matrix power $A^{\otimes g(H)}$.
Denote by $A_{k\cdot}$ the $k$\textsuperscript{th} row of $A$.

\begin{lemma}[Nachtigall~\cite{Nacht}]\label{lem:nacht}
Let~$k$ be a critical node on a critical cycle of length~$\ell$.
Then $T_k\leq (n-1)\cdot \ell$ and~$\ell$ is an eventual period
of~$A^{\otimes t}_{k\cdot}$
\end{lemma}
\begin{proof}
Set $B=A^{\otimes \ell}$. Then $b_{k,k}=1$ and hence
$b_{k,j}^{(t)}$ is non-decreasing with~$t$.
But since we assume~$\lambda(A)=1$, $\lambda(B)=1$ and $\sup_{t\in\mathbb{N}}B^{\otimes t}=\max_{t=1}^{n-1}B^{\otimes t}$,
so $b_{k,j}^{(t)}$ is constant for~$t\ge n-1$, i.e. $T_k(A^{\otimes \ell})\le n-1$.

Lemma~\ref{lem:trans:multiples} now concludes the proof.
\end{proof}

The following result enables us to use the bound of Lemma~\ref{lem:nacht}
for nodes that do not lie on a critical cycle of minimal length. 
As usual, we assume that $\lambda(A)=1$ and $A$ is visualized. 
The proof makes heavy use of this assumption.
It is closely related to the circulant symmetries of the critical part of
 max-algebraic powers in the periodic regime, as described by Butkovi\v{c} and Sergeev~\cite[Section 8.3]{But:10},
 \cite{Ser-09}.

\begin{lemma}
\label{lem:trans:walk}
Let $k$ and $l$ be two indices of $N_c(A)$,
and suppose that there exists a walk from $k$ to $l$, of length $r$ and with all edges critical.
\begin{itemize}
\item[{\rm (i)}] If $t\geq T_l(A)$,  then $A_{k\cdot}^{\otimes (t+r)}=\lambda(A)^r\cdot A_{l\cdot}^{\otimes t}$.
\item[{\rm (ii)}] $T_k(A)\leq T_l(A)+r$.
\end{itemize}
\end{lemma}

\begin{remark}\label{r:remHperiod}
{\rm Since none of the above lemmas assume the irreducibility
of~$A$, they imply the eventual periodicity of all rows and columns
with critical indices in the reducible case. Moreover, they show
that the cyclicity~$\gamma(H)$ of a strongly connected component~$H$
of~$\crit(A)$ is an eventual period for the~$A^{\otimes t}_{k\cdot}$
for any~$k\in H$. It is the least period because when~$A$ is
strictly visualized, $a_{k,k}^{(t)}$ takes the value~$1$ with least
eventual period~$\gamma(H)$. }
\end{remark}

\begin{proof}[Proof of Lemma~\ref{lem:trans:walk}]

Since $A$ is assumed to be visualized with~$\lambda(A)=1$, the existence of the walk with critical edges exactly means~$a_{k,l}^{(r)}=1$.

Since each edge of $\crit(A)$ belongs to a cycle of $\crit(A)$, there is a walk from~$l$ to~$k$ with critical edges.
Let~$s\ge 1$ be its length. We have~$a_{l,k}^{(s)}=1$.

Thus, we have
$$a_{l,j}^{(t+r+s)} \geq a_{l,k}^{(s)}.a_{k,j}^{(t+r)} = a_{k,j}^{(t+r)} \geq a_{k,l}^{(r)}.a^{(t)}_{l,j}=a^{(t)}_{l,j}$$ for any~$t$.

Iterating the inequality, we see that
\begin{equation}\label{eq:lem:trans:walk}
a_{l,j}^{(t+p(r+s))} \geq a_{k,j}^{(t+r)} \ge a^{(t)}_{l,j}
\end{equation}
for all~$t$ and~$p$.

If $p$ is an eventual period of~$a^{(t)}_{l,j}$ and $t\ge T_l$, the first and the last entry of~\eqref{eq:lem:trans:walk} are equal, so all inequalities of~\eqref{eq:lem:trans:walk} are equalities.

It means that the sequences
$\left(a^{(t)}_{k,j}\right)_{t\ge T_l+r},\quad\left(a^{(t-r)}_{l,j}\right)_{t\ge T_l+r},\quad\left(a^{(t)}_{l,j}\right)_{t\ge T_l}$
are identical. Since the last sequence is periodic, both parts of the lemma are proved.
\end{proof}

\begin{proof}[Proof of Dulmage-Mendelsohn bound]
Let~$\cycle$ be a cycle in~$H$ of length $\ell(\cycle) = g(H)$. By
Lemma~\ref{lem:nacht}, $T_k\leq (n-1) \cdot g(H)$ for all
nodes~$k$ of~$\cycle$.

Let now~$k$ be any node in~$H$. There exist walks in~$H$ from~$k$
to~$\cycle$ of length at most~$\lvert H \rvert -
g(H)$. Application of Lemma~\ref{lem:trans:walk} now
concludes the proof.
\end{proof}

\section{Proof of Kim's Bound}\label{s:Kim}
Set $B=A^{\otimes d}$. The cyclicity classes of~$\digr(A)$ are
the strongly connected components  of~$\digr(B)$,
and $\digr(B)$ is completely
reducible, i.e. it has no edge between two different strongly connected
components.

It means that, up to reordering the indices, $B$ is block-diagonal.
As in Kim~\cite{Kim-79}, we want to apply the previous bound to the smallest blocks.
Then, we will get the general bound thanks to Lemmas~\ref{lem:trans:multiples} to~\ref{lem:trans:walk}.

Obviously, any cycle in~$\digr(A)$ has to go through every cyclicity class.
Thus, $d$ divides $g(H)$ and if $k$ belongs to~$H$ then  the girth of its
strongly connected components in~$\crit(B)$ is at most $g(H)/d$.

Call a cyclicity class of~$\digr(A)$ {\em small\/} if it contains the
minimal number of nodes amongst cyclicity classes. Let~$m$ be the
number of nodes in any small class.

We apply the Dulmage-Mendelsohn bound to the nodes of small classes and then use
Lemma~\ref{lem:trans:walk} to extend it to all other critical nodes.

We distinguish the cases (A) $m\leq \lfloor n/d\rfloor -1$ and (B)
$m= \lfloor n/d\rfloor$. Note that $m\geq \lfloor n/d\rfloor+1$ is
not possible.

In case (B), there are at least $d-(n\bmod d)$ small classes because
otherwise the sum of sizes of cyclicity classes~$C$ would satisfy
\begin{equation}\label{e:small:classes}
\begin{split}
n =\sum_C \lvert C\rvert & > \big(d-(n\bmod d)\big) \cdot \lfloor n/d\rfloor +
(n\bmod d)
\cdot \big( \lfloor n/d \rfloor + 1 \big)
\\
& = d\cdot \lfloor n/d \rfloor + (n\bmod d) = n
\enspace,
\end{split}
\end{equation}
a contradiction. Hence every critical node is connected to a small class by a path consisting of
critical edges of length at most~$(n\bmod d)$.

Let us first prove that
\begin{equation}\label{eq:kim1}
T_k\le \big(\lfloor n/d\rfloor - 2\big) \cdot g(H) + n.
\end{equation}

In both cases (A) and (B), by the max-algebraic extension of  Dulmage and Mendelsohn's bound,
we have $T_k(B)\leq (m-2)\cdot g(H)/d + m$
for each critical node~$k$ of~$H$ in a small class.
Lemma~\ref{lem:trans:multiples} then
implies $T_k(A)\leq (m-2) \cdot g(H) + d\cdot m$ for all critical
nodes~$k$ of~$H$ in small classes.

In case (A), a crude estimation for all critical~$k$ in small
classes is
\[
T_k\leq
(m-1) \cdot g(H) + d\cdot m \leq  \big(\lfloor n/d\rfloor - 2\big)
\cdot g(H)
+ n - d
\enspace.
\]
Because every critical node has paths consisting of critical edges
to a small class of length at most~$d-1$, \eqref{eq:kim1}
follows from Lemma~\ref{lem:trans:walk} in case (A).

In case (B), there is a path of length at
most~$(n\bmod d)$ consisting of
critical edges to a small class. Hence, again by Lemma~\ref{lem:trans:walk},
\begin{align*}
T_k & \leq \big(\lfloor n/d\rfloor - 2\big) \cdot g(H) + d\cdot \lfloor
n/d\rfloor  + (n \bmod d)
\\
& = \big(\lfloor n/d\rfloor - 2\big) \cdot g(H) + n
\enspace.
\end{align*}
This concludes the proof of~\eqref{eq:kim1}.

It remains to prove
\begin{equation}\label{eq:kim2}
T_k(A)\le \big(\lfloor n/d\rfloor - 2\big) \cdot g(H) + \lvert H\rvert +
(n\bmod d).
\end{equation}
This is based on the following lemma, which we prove using the same case
distinction.

\begin{lemma}\label{l:kim2}
 If $k$ is in a critical closed walk~$\cycle$, then
\begin{equation}\label{e:l:kim2}
 T_k(A)\le (\lfloor n/d\rfloor -1)\ell(C) + (n \bmod d),
\end{equation}
\end{lemma}
\begin{proof}
Let us first notice that~$k$ is in a critical closed walk of~$A^{\otimes d}$ with length~$\ell(C)/d$.

If $k$ is in a small class, then
$T_k(A^{\otimes d}) \le (m -1)\ell(C)/d$ by Lemma~\ref{lem:nacht}, thus $T_k(A)\le (m -1)\ell(C)$ by Lemma~\ref{lem:trans:multiples}.

If $k$ is in~$\cycle$ but not necessarily in a small class, we distinguish between cases (A)~and~(B) and apply Lemma~\ref{lem:trans:walk}.
In case~(A), $m\le \lfloor n/d\rfloor -1$. Recall that $C$ contains representatives of all cyclicity classes
(and the small classes, too), and therefore, each node $k\in C$ can be connected to
a node of $C$ in a small class by a subpath of $C$, with length at most $\ell(C)-1$.
Applying Lemma~\ref{lem:trans:walk} we get $$T_k(A)\le (\lfloor n/d\rfloor -2)\ell(C) + \ell(C)-1=(\lfloor n/d\rfloor -1)\ell(C) -1$$
which implies~\eqref{e:l:kim2}.
In case~(B), $m= \lfloor n/d\rfloor$ and there is a path from~$k$ to a node of~$\cycle$ in a small class with length at most~$(n\bmod d)$,
so we get exactly~\eqref{e:l:kim2}.
\end{proof}

To conclude the proof of~\eqref{eq:kim2}, we apply the Lemma~\ref{l:kim2} to a cycle~$C$ with length~$g(H)$
and notice that for any~$k\in H$, there is a critical path from~$k$
to~$\cycle$ with length at most $\lvert H\rvert-g(H)$.
By Lemma~\ref{lem:trans:walk}, it implies

\begin{align*}
 T_k(A) &\le (\lfloor n/d\rfloor -1)g(H) +  (n \bmod d) + \lvert H\rvert-g(H)\\
&=(\lfloor n/d\rfloor -2)g(H)+\lvert H\rvert+ (n \bmod d)
\end{align*}

and~\eqref{eq:kim2} is proved.

\section{Proof of Wielandt's Bound}
\label{s:wielandt}

If $g(H)\leq n-1$, then Wielandt's bound follows from the
Dulmage-Mendelsohn bound. It remains to treat the case that
$g(H)=n$, i.e. $\crit(A)$ is a Hamiltonian cycle. We therefore prove
a result on cycle removal and insertion
(Theorem~\ref{PumpingLemL=n}) which implies the Wielandt bound for
matrices with a critical Hamiltonian cycle
(Corollary~\ref{c:WilHamil}).

\subsection{Cycle Replacement with a Hamiltonian Cycle}

We recall the following elementary application of the pigeonhole
principle.

\begin{lemma}
\label{l:NT} Let $x_1,\ldots, x_n$ be integers.
There exists a nonempty subset~$I$ of $\{1,\dots,n\}$
such that $\sum_{i\in I} x_i$ is a
multiple of $n$.
\end{lemma}
\begin{proof}
Either one of the~$n$ sums $\sum_{i=1}^k x_i$ with $1\leq k\leq n$ is a
multiple of~$n$, or two of these sums are in the same congruence class
modulo~$n$.
\end{proof}

We will use this lemma to prove:

\begin{theorem}\label{PumpingLemL=n}
Let $\digr$ be a digraph with $n$ nodes. For any Hamiltonian
cycle~$\cycle_H$ in $\digr$ and any walk $W$, there is a walk~$V$ that
has the same start and end node as $W$, is formed by removing cycles from
$W$ and possibly inserting copies of $\cycle_H$, and has a length
satisfying $(n-1)^2+1\leq \ell(V)\leq (n-1)^2+n$ and $\ell(V)\equiv \ell(W)
\pmod n$.
\end{theorem}

\begin{corollary}\label{c:WilHamil}
If $A\in\Rpnn$ has a critical Hamiltonian cycle and~$\lambda(A)=1$, then the transient of~$A$ is at most~$\wiel(n)$.
\end{corollary}

\begin{proof}[Proof of Theorem~\ref{PumpingLemL=n}]
$W$ can be decomposed into a path $P$ and a collection~$\mathcal{C}$ of cycles.
Note that~$P$ is empty when the start and the end nodes of $W$ are the same.

Let~$\mathcal{B}$ be a result of recursively removing from~$\mathcal{C}$ sets of cycles whose combined length is a multiple
of~$n$. By Lemma~\ref{l:NT}, $\lvert \mathcal{B}\rvert  \leq n-1$. Also,
$\ell(C)\leq n-1$ for all $C\in \mathcal{B}$.

Let us build the walk~$V$ as follows. If~$P$ intersects all cycles of~$\mathcal{B}$ [Case (C)], then we successively insert all such cycles in~$P$.
Otherwise [Case (D)], we first insert~$\cycle_H$ into~$P$, getting~$\tilde{P}$ and then insert all cycles of~$\mathcal{B}$ into~$\tilde{P}$.

In case~(C), we have
\begin{align*}
\ell(V) & = \ell(P) + \sum_{\alpha\in R} \ell(C_\alpha)
\\
& \leq (n-1) + (n-1)\cdot (n-1)
\\
& < (n-1)^2 + n
\end{align*}

In case~(D), there exists some $\hat{C}\in\mathcal{B}$ such that $\ell(P) + \ell(\hat{C}) \leq n-1$, so that
\begin{align*}
\ell(V) & = \ell(\cycle_H) + \ell(P) + \ell(\hat{C}) + \sum_{\substack{C\in\mathcal{B}\\ C \neq \hat{C}}} \ell(C)
\\
& \leq n + (n - 1) + (n-1)\cdot (n-2)
\\
& = (n-1)^2 + n
\end{align*}

Moreover, $\ell(V) \equiv \ell(W) \pmod n$ by construction in both cases.

This concludes the proof, because if $V$ is too short, we just insert copies of~$\cycle_H$ into it.
\end{proof}

\subsection{Proof of  Wielandt's Bound with a Critical Hamiltonian Cycle}

In this section, we prove Corollary~\ref{c:WilHamil}.
Because the critical graph contains a Hamiltonian cycle, it is strongly connected and $n$ is an
eventual period of~$A^{\otimes t}$.

Let $i,j\in N$ and let~$t\geq \wiel(n)$.
We show that $a_{i,j}^{(t)} = a_{i,j}^{(s(t))}$ where $s(t) = \wiel(n) +
\big((t - \wiel(n)) \bmod n\big)$.
Because $s(t) = s(t')$ whenever $t\equiv t' \pmod n$, this suffices for the
proof.

If~$t=s(t)$ the result is obvious.
Otherwise, let~$W$ be a maximum weight walk of length~$t$ from~$i$ to~$j$, i.e. $p(W) =
a_{i,j}^{(t)}$.
We apply Theorem~\ref{PumpingLemL=n} to walk~$W$ and the critical Hamiltonian
cycle~$\cycle_H$.
By Theorem~\ref{PumpingLemL=n}, there is a walk $V$
from $i$ to $j$, obtained from $W$ by deleting some cycles and
possible inserting copies of $\cycle$, with
length satisfying $\wiel(n)\leq \ell(V)\leq \wiel(n)+n-1$ and $\ell(V)\equiv
\ell(W)\pmod n$.
In other words, $\ell(V) = s(t)$.
Because we assume $\lambda(A)=1$, the weight of~$V$ satisfies $p(V)\geq p(W)$
and hence
\begin{equation}\label{e:wiel:proof:1}
a_{i,j}^{(s(t))} \geq p(V) \geq p(W) = a_{i,j}^{(t)}
\enspace.
\end{equation}
Since $t\geq \wiel(n)$ and $s(t)\equiv t \pmod n$, there exists some
$r\geq0$ such that $t-s(t) = r\cdot n$.
Hence
\begin{equation}\label{e:wiel:proof:2}
a_{i,j}^{(t)} \geq a_{i,j}^{(s(t))} \cdot a_{j,j}^{(r\cdot n)} =
 a_{i,j}^{(s(t))}
\end{equation}
because $a_{j,j}^{(n)} = 1$.
Combination of \eqref{e:wiel:proof:1} and \eqref{e:wiel:proof:2} concludes the
proof.

\begin{remark}
\label{r:HA}
{\rm To our knowledge, the results of this section are new. However,  let us remark that the method
of cycle replacement using Lemma~\ref{l:NT} was invented by Hartmann and Arguelles~\cite{HA-99}, who also
used it to derive (less precise) transience bounds for sequences of optimal walks.}
\end{remark}

\section{Proof of Schwarz's Bound}
\label{s:Sch}
Schwarz's bound is derived from Wielandt's bound as Kim's bound is derived from Dulmage-Mendelsohn's bound in Section~\ref{s:Kim}.

Again, call a cyclicity class of~$\digr(A)$ {\em small\/} if it contains the
minimal number of nodes amongst cyclicity classes. Let~$m$ be the
number of nodes in any small class and set $B=A^{\otimes d}$.

For each critical node~$k$ in a small cyclicity class,
we have
$T_k(B)\leq \wiel( m)$ by Wielandt's bound.
Lemma~\ref{lem:trans:multiples} hence implies $T_k(A)\leq d
\cdot \wiel(m )$ for all critical nodes~$k$ in small classes.

We distinguish the cases (A) $m\leq \lfloor n/d\rfloor -1$ and (B)
$m= \lfloor n/d\rfloor$. Note that $m\geq \lfloor n/d\rfloor+1$ is
not possible.

In case (A), a crude estimation for all critical~$k$ in small
classes is
\[
T_k\leq d\cdot\wiel(m) \leq d\cdot\wiel\big(\lfloor n/d\rfloor\big) - d
\enspace.
\]
Observe that each critical node can be connected to a small class, by a path of length
at most $d-1$ consisting of critical edges only. So in case (A), the theorem just
follows from Lemma~\ref{lem:trans:walk}.

In case (B), there are at least $d-(n\bmod d)$ small classes because
otherwise~\eqref{e:small:classes} yields a contradiction.
In this case, each critical node can be connected to
a node from a small class by a path consisting only of critical
edges, of length at
most~$(n\bmod d)$.
Hence, by Lemma~\ref{lem:trans:walk},
\begin{align*}
T_k \leq d\cdot \wiel\big( \lfloor n/d\rfloor \big) + (n\bmod d)
\enspace.
\end{align*}
This concludes the proof.

\section{Proof of the Bounds Involving the Factor Rank}\label{s:rank}

In this section, we prove Main Theorem~\ref{t:rank}.
Let $v_{\alpha}, w_{\alpha}\in\Rp^n$, for $\alpha=1,\ldots,r$, be the vectors in factor rank
representation~\eqref{e:rank}. Further, Let $V$ and $W$ be the $n\times r$ matrices whose columns are vectors
$v_{\alpha}$ and $w_{\alpha}$ for $\alpha=1,\ldots,r$,
and consider the $(n+r)\times(n+r)$ matrix~$Z$ defined by
\begin{equation}
\label{e:zdef}
Z=
\begin{pmatrix}
0_{n\times n} & V\\
W^T & 0_{r\times r}
\end{pmatrix},
\end{equation}


Then we have
\begin{equation}
\label{e:z-a-b}
Z^{\otimes 2}=
\begin{pmatrix}
A & 0_{n\times r}\\
0_{r\times n} & B
\end{pmatrix},
\end{equation}
where the $r\times r$ matrix $B$ is given by
\begin{equation}
\label{e:bdef}
b_{\alpha,\beta} = \bigoplus_{i=1}^n w_{\alpha,i} \cdot
v_{\beta,i}, \  \text{for $\alpha,\beta=1,\ldots,r$}.
\end{equation}

We will apply the bounds of Main Theorem~\ref{t:mainres} to the critical nodes
of~$B$ and transfer the result to the critical nodes of~$A$, thanks to the following observation.

\begin{lemma}\label{l:transfer}
If~$(k,n+\beta)$ is an edge of~$\crit(Z)$, then $T_k(A)\le T_\beta(B)+1$.
\end{lemma}
\begin{proof}
By Equation~\eqref{e:z-a-b} and Lemma~\ref{lem:trans:multiples},
$T_{n+\beta}(Z)\leq 2T_{n+\beta}(Z^{\otimes 2})\leq 2T_\beta(B)$, thus $T_k(Z)\leq 2T_\beta(B)+1$ by
Lemma~\ref{lem:trans:walk}.
But Equation~\eqref{e:z-a-b} now implies
$T_k(A)\leq \lceil T_k(Z)/2\rceil  \leq\lceil (2T_\beta(B)+1)/2 \rceil  = T_\beta(B)+1$.
\end{proof}

To use this lemma, we need to study the links between~$\crit(Z)$, $\crit(A)$ and~$\crit(B)$.

The next observation is useful for the case of Kim's and Schwarz's bounds, where $A$
is assumed to be irreducible.
\begin{lemma}\label{l:ABZirred}
If $A$ is irreducible, then so are~$Z$ and~$B$. Moreover~$\digr(B)$ and~$\digr(A)$ have the same cyclicity.
\end{lemma}

\begin{proof}
As $A$ is irreducible, there exists a walk in~$\digr(Z^{\otimes 2})$, and hence in $\digr(Z)$,
between every pair of nodes in $\{1,\dots,n\}$.
None of the vectors $v_\alpha$,
$w_\alpha$ for $\alpha=1,\ldots,r$ is zero by the minimality of~$r$, i.e. every node in
$\{n+1,\dots,n+r\}$ has an incoming and an outgoing neighbor in
$\{1,\dots,n\}$.
Hence there exists a walk between every pair of nodes in~$\digr(Z)$.

By Theorem~\ref{t:graphpowers}, $\digr(Z^{\otimes 2})$ has at most~$2$ strongly connected components with the same cyclicity.
By Equation~\eqref{e:z-a-b}, it has at least~$2$ components, one of them is~$\digr(A)$ and the second one is isomorphic
to~$\digr(B)$, hence these are the two components given by Theorem~\ref{t:graphpowers}.
In particular, $B$ is irreducible and the graphs of~$A$ and~$B$ have the same cyclicity.
\end{proof}



By construction, $\digr(Z)$ is a bipartite graph, so every walk in~$\digr(Z)$
alternates between nodes in $\{1,\dots,n\}$ and nodes
in $\{n+1,\dots,n+r\}$. Figure~\ref{f:z} depicts an example of a walk in~$\digr(Z)$.

\begin{figure}
\centering
\begin{tikzpicture}[>=latex,scale=1.0,very thick]
\node[shape=circle,draw] (i) at (0.0,3.0) {$i$};
\node[shape=circle,draw] (j) at (3.0,3.0) {$j$};
\node[shape=circle,draw] (k) at (6.0,3.0) {$k$};
\node[shape=circle,draw] (l) at (9.0,3.0) {$l$};

\node[shape=circle,draw] (a) at (1.5,0.0) {$\scriptstyle n+\alpha$};
\node[shape=circle,draw] (b) at (4.5,0.0) {$\scriptstyle n+\beta$};
\node[shape=circle,draw] (c) at (7.5,0.0) {$\scriptstyle n+\gamma$};

\draw[->] (i) to node[left] {$v_{\alpha,i}$} (a);
\draw[->] (a) to node[left] {$w_{\alpha,j}$} (j);
\draw[->] (j) to node[left] {$v_{\beta,j}$} (b);
\draw[->] (b) to node[left] {$w_{\beta,k}$} (k);
\draw[->] (k) to node[left] {$v_{\gamma,k}$} (c);
\draw[->] (c) to node[left] {$w_{\gamma,l}$} (l);
\end{tikzpicture}
\caption{A walk in $\digr(Z)$}
\label{f:z}
\end{figure}

\if{
Without loss of generality, we assume that $\lambda(Z)=1$ and~$Z$ is strictly visualized.
Thus, critical edges are exactly the ones with weight~$1$.

Each closed walk of $\crit(Z)$  contains nodes both from $\{1,\ldots,n\}$ and from $\{n+1,\ldots,n+r\}$.
Thus it arises from a closed walk on~$\digr(A)$ and from a closed walk on~$\digr(B)$, all of them have only edges with weight~$1$.
It implies that $\lambda(B)=\lambda(A)=1$ and both closed walks are critical for~$A$ and~$B$ respectively.

Now, every critical edge of~$\crit(A)$ or~$\crit(B)$ has weight~$1$, so it results from a path of $\digr(Z)$ of length $2$ with weight~$1$,
which has only critical edges since~$A$ is strictly visualized. Inserting these paths into a closed walk~$C$ on~$\crit(A)$ or~$\crit(B)$, we obtain a closed
walk of $\crit(Z)$ (see Figure~\ref{f:correspondence}, left).
In $\digr(Z^{\otimes 2})$ this walk splits in two closed walks of $\crit(Z^{\otimes 2})$ with length~$\ell(C)$,
one of them is~$C$ and the otherone~$\tilde{C}$ lives on the other critical graph.

But  components of critical graphs are made of walks. Hence each component $G$ of $\crit(Z)$ splits into two components of $(\crit(Z))^2$ such that the (disjoint) union of
their node sets is exactly the node set of $G$.  Following~\cite{BSST} we call
these two components {\em related}. For a component $H$ of $(\crit(Z))^2$, the related component
will be denoted by $H'$.

If $(H,H')$ is a pair of related components of $(\crit(Z))^2$ then one of them (say, $H$) contains a node in $\{1,\ldots,n\}$ and the other
($H'$) contains a node in $\{n+1,\ldots,n+r\}$. Since there are no edges between the two components of
$\digr(Z^{\otimes 2})$, $H$ is a subgraph of $\digr(A)$ and $H'$ is a subgraph of $\digr(B)$.
Further as $(\crit(Z))^2=\crit(Z^{\otimes 2})$ (by Lemma~\ref{l:CAk}),
$H$ is a component of $\crit(A)$, and $H'$ is a component of $\crit(B)$. Moreover, $\crit(A)$ and $\crit(B)$
do not have components that are not formed this way.
}\fi

\if{
Let us apply Theorems~\ref{t:graph powers},~\ref{t:girth} and~\ref{t:maxpowers} to
$Z$. In our case, $k=2$, and $d$ and $\sigma$ are even, therefore $Z^{\otimes 2}$ is the
direct sum of two irreducible blocks $A$ and $B$ with $\lambda(A)=\lambda(B)$ and the
same cyclicity $d/2$.  Each component of $\crit(Z)$ gets split into
two components of $\crit(Z^{\otimes 2})=\crit(Z)^2$. One of these is a component of $\crit(A)$, and the other
is a component of $\crit(B)$. Following~\cite{BSST}, we call these components {\em related}.
}\fi

\if{
Proposition~\ref{p:powers} can be explained in terms of the correspondence between the closed walks
of $\digr(A)$, $\digr(B)$ and $\digr(Z)$, as depicted on Figure{f:correspondence}.
Observe that if one of these walks is critical, then the two other walks can be made critical as well,
so this is also a correspondence between the critical closed walks.
}\fi

As all closed walks in $\digr(Z)$ are of even length, the cyclicity
of any component of $\crit(Z)$ is even, i.e. it is divisible by two.
Hence each component $G$ of $\crit(Z)$ splits into two components of $(\crit(Z))^2$ such that the (disjoint) union of
their node sets is exactly the node set of $G$
(e.g., apply Theorem~\ref{t:graphpowers} with $k=2$ and even $\sigma$).  Following~\cite{BSST} we call
these two components {\em related}. For a component $H$ of $(\crit(Z))^2$, the related component
will be denoted by $H'$.

Each closed walk
of $\digr(Z)$ and, therefore, each component of $\crit(Z)$ contains
nodes both from $\{1,\ldots,n\}$ and from $\{n+1,\ldots,n+r\}$.
Hence, if $H$ and $H'$ is a pair of
related components of $(\crit(Z))^2$ then one of them (say, $H$) contains a node in $\{1,\ldots,n\}$ and the other
($H'$) contains a node in $\{n+1,\ldots,n+r\}$. Since there are no edges between the two components of
$\digr(Z^{\otimes 2})$, $H$ is a subgraph of $\digr(A)$ and $H'$ is a subgraph of $\digr(B)$.
Further as $(\crit(Z))^2=\crit(Z^{\otimes 2})$ (by Lemma~\ref{l:CAk}),
$H$ and $H'$ are components of $\crit(Z^{\otimes 2})$. As $\crit(Z^{\otimes 2})$ consists of only such
components and the cycles not belonging to such components have a strictly smaller geometric mean, it follows
that $H$ is a component of $\crit(A)$, $H'$ is a component of $\crit(B)$ and, moreover,
$\crit(A)$ and $\crit(B)$ do not have components that are not formed this
way.\footnote{In fact, we also have
$\lambda(A)=\lambda(B)=\lambda(Z^{\otimes 2})=(\lambda(Z))^2$.}

\begin{lemma}
\label{l:related}
Let $H$ be a component of $\crit(A)$.
\begin{itemize}
\item[{\rm (i)}]  $g(H)=g(H')$.
\item[{\rm (ii)}] If~$k$ belongs to a closed walk~$\cycle$ on~$H$,
then there are edges $(n+\alpha,k)$ and $(k,n+\beta)\in\crit(Z)$,
such that~$\alpha$ and~$\beta$ belong to a closed walk~$\tilde{\cycle}$ in~$H'$, with~$\ell(\cycle)=\ell(\cycle')$.
\end{itemize}
\end{lemma}
\begin{proof}
Take a closed walk~$C$ on $H$. Each edge of~$C$ results from a path of
$\crit(Z)$ of length $2$, and inserting these path in~$C$ we obtain a closed
walk of $\crit(Z)$ (see Figure~\ref{f:correspondence}, left). This
walk contains nodes from both $H$ and $H'$. In $Z^{\otimes 2}$ it
splits in two closed walks of $\crit(Z^{\otimes 2})$ of the same length (see
Figure~\ref{f:correspondence}, right). One of these closed walks is~$C$ and the other is a closed walk~$\tilde{C}$ of $H'$
(since $H$ and $H'$
are isolated in $\crit(Z^{\otimes 2})$).

This implies $g(H')\leq g(H)$, and the
reverse inequality follows by symmetry, hence part (i).\footnote{A similar argument shows that for any
strongly connected graph $\digr$, all components of $\digr^k$ have
the same girth.} It also follows
that each node of the original cycle in $H$ has neighbors belonging
to a closed walk in $H'$. Since each node of $H$ lies on a cycle, we have part (ii).
\end{proof}

\begin{figure}
\centering
\begin{tikzpicture}[>=latex,scale=0.8,very thick]
\begin{scope}[shift={(0,0)}]
\node[shape=circle,draw] (i1) at ( 90:3.0) {$i_1$};
\node[shape=circle,draw] (i2) at (162:3.0) {$i_2$};
\node[shape=circle,draw] (i3) at (234:3.0) {$i_3$};
\node[shape=circle,draw] (i4) at (306:3.0) {$i_4$};
\node[shape=circle,draw] (i5) at ( 18:3.0) {$i_5$};

\node[shape=circle,draw] (a1) at (126:1.5) {$\scriptstyle n+\alpha_1$};
\node[shape=circle,draw] (a2) at (198:1.5) {$\scriptstyle n+\alpha_2$};
\node[shape=circle,draw] (a3) at (270:1.5) {$\scriptstyle n+\alpha_3$};
\node[shape=circle,draw] (a4) at (342:1.5) {$\scriptstyle n+\alpha_4$};
\node[shape=circle,draw] (a5) at ( 54:1.5) {$\scriptstyle n+\alpha_5$};

\draw[->] (i1) to  (a1);
\draw[->] (a1) to  (i2);
\draw[->] (i2) to  (a2);
\draw[->] (a2) to  (i3);
\draw[->] (i3) to  (a3);
\draw[->] (a3) to  (i4);
\draw[->] (i4) to  (a4);
\draw[->] (a4) to  (i5);
\draw[->] (i5) to  (a5);
\draw[->] (a5) to  (i1);

\node at (200:3.2) {$\digr(Z)$};
\end{scope}

\draw[|->] ( 18:4) -- ( 18:6);
\draw[|->] (-18:4) -- (-18:6);

\begin{scope}[shift={(18:8.5)}]
\node[shape=circle,draw] (i1) at ( 90:1.5) {$i_1$};
\node[shape=circle,draw] (i2) at (162:1.5) {$i_2$};
\node[shape=circle,draw] (i3) at (234:1.5) {$i_3$};
\node[shape=circle,draw] (i4) at (306:1.5) {$i_4$};
\node[shape=circle,draw] (i5) at ( 18:1.5) {$i_5$};

\draw[->] (i1) to (i2);
\draw[->] (i2) to (i3);
\draw[->] (i3) to (i4);
\draw[->] (i4) to (i5);
\draw[->] (i5) to (i1);

\node at (-20:2.4) {$\digr(A)$};
\end{scope}

\begin{scope}[shift={(-18:8.5)}]
\node[shape=circle,draw] (a1) at (126:1.5) {$\alpha_1$};
\node[shape=circle,draw] (a2) at (198:1.5) {$\alpha_2$};
\node[shape=circle,draw] (a3) at (270:1.5) {$\alpha_3$};
\node[shape=circle,draw] (a4) at (342:1.5) {$\alpha_4$};
\node[shape=circle,draw] (a5) at ( 54:1.5) {$\alpha_5$};

\draw[->] (a1) to (a2);
\draw[->] (a2) to (a3);
\draw[->] (a3) to (a4);
\draw[->] (a4) to (a5);
\draw[->] (a5) to (a1);

\node at (20:2.4) {$\digr(B)$};
\end{scope}
\end{tikzpicture}
\caption{Correspondence between closed walks of $\digr(Z)$, $\digr(A)$ and $\digr(B)$.}
\label{f:correspondence}
\end{figure}
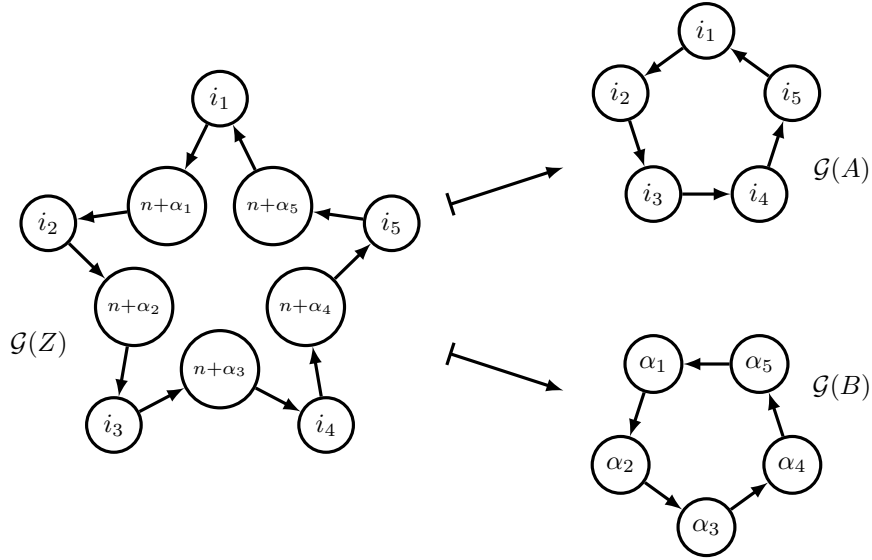

\if{
Every critical closed walk in~$\digr(Z)$ of length~$\ell$
gives rise to both a closed walk in~$\digr(A)$ of length~$\ell/2$ and a
closed walk in~$\digr(B)$ of length~$\ell/2$. This correspondence, depicted
on Figure~\ref{f:correspondence}, is surjective, and it can be observed that
if one of the closed walks consists of critical edges only, then so do
In particular, the cyclicities of~$\digr(A)$ and~$\digr(B)$ are equal, i.e.
also $\digr(B)$ has cyclicity~$d$.
The correspondence also maps critical closed walks to critical closed walks.
That is, $k$ is also a critical node in~$\digr(Z)$ and it has both a critical
incoming neighbor $n+\alpha$ and a critical outgoing neighbor $n+\beta$ in
$\crit(Z)$.
Furthermore, both~$\alpha$ and~$\beta$ are critical in~$\digr(B)$ and they are
contained in the same component~$H'$ of $\crit(B)$.
The correspondence also yields the equality $g(H') = g(H)$ of the girths.
}\fi

Define
\begin{equation}
\label{e:hdef}
h=\min(\lvert H\rvert,\lvert H'\rvert).
\end{equation}
We are ready for the proof of Main Theorem~\ref{t:rank}.

\begin{proof}[Proof of Main Theorem~\ref{t:rank}]
Let $Z$ and $B$ be the matrices defined in~\eqref{e:zdef} and~\eqref{e:bdef} .
Let  $k$ be an index in $H$ (belonging to $\{1,\ldots,n\}$).
By Lemma~\ref{l:related} part (ii), there is an edge of $\crit(Z)$
connecting it to some node $n+\beta$, for $\beta\in\{1,\ldots,r\}$, which belongs to $H'$.
For each bound of Main~Theorem~\ref{t:mainres}, an application
 Lemma~\ref{l:transfer} yields a
version of the corresponding bound of Main Theorem~\ref{t:rank} on $T_k$,
where $d$ is the cyclicity of~$B$, and we have $g(H')$ instead of $g(H)$ and
$\lvert H'\rvert$ instead of $h$.

However, $g(H')=g(H)$ and $\digr(B)$ has the same cyclicity as~$\digr(A)$ (when $A$ and hence also $B$ are
irreducible), so it only remains
to explain why we have $h$ (and not $\lvert H'\rvert$), in the factor rank versions of the
bounds of Dulmage and Mendelsohn, and Kim. The following argument accounts for both cases (set $d=1$ for
Dulmage and Mendelsohn's bound).

First, let $k$ belong to a cycle~$\cycle$ with length~$g(H)$.
By Lemma~\ref{l:related} part~(i), $\beta$ belongs to a closed walk~$\tilde{C}$ with length~$g(H)=g(H')$ in~$H'$,
so we can apply Lemma~\ref{l:kim2} to~$\beta$ and get~$T_\beta(B)\le
g(H)(\lfloor r/d \rfloor-1) + (r \bmod d)$ (for the
bound of Kim), or apply Lemma~\ref{lem:nacht} and get~$T_\beta(B)\le g(H)(r-1)$ (for the bound of Dulmage and Mendelsohn).
By Lemma~\ref{l:transfer}, we get
$$T_k(A) \le g(H)(\lfloor r/d \rfloor-1) + (r \bmod d)+1.$$

Second, if $k$ does not belong to such a cycle,
we can apply Lemma~\ref{lem:trans:walk},
because~$k$ is connected to~$\cycle$ by a path on~$H$ of length at
most~$\lvert H\rvert-g(H)$. Hence we obtain
$$T_k(A) \le g(H)(\lfloor r/d \rfloor-2) + (r \bmod d)+\lvert H\rvert+1$$
and the factor rank versions of Kim's and Dulmage-Mendelsohn's bounds. The proof is complete.
\end{proof}

\section{On the Precision of the Bounds}
\label{s:precision}
Since the bounds of Main Theorem~\ref{t:mainres} are extensions of the bounds on Boolean matrices
and the latter are known to be exact (see for instance~\cite{LS-93} and the references therein), so are the bounds of Main Theorem~\ref{t:mainres}. However,
the max-algebraic case is richer, and some natural questions arise. A general question is
when these bounds are attained.  To begin with, can these bounds be
attained by the matrices whose critical graph does not attain the corresponding ``Boolean'' bound, or can the bounds
be attained when not all the nodes are critical.

The easiest way to produce max-plus examples from Boolean ones is to
use the semigroup morphism~$\phi_0:\Rpnn\rightarrow\bool^{n\times
n}$ that maps~$A$ to its  {\em pattern} $B=\phi_0(A)$ such that
$b_{i,j}=0$ if and only if $a_{i,j}=0$ and $b_{i,j}=1$ otherwise. Since it is a
morphism, we have:

\begin{lemma}
\label{l:pattern}
Let $A,B\in\Rpnn$ have the same pattern and let $A$ be Boolean.
Then $T_k(A)\leq T_k(B)$ for all $k=1,\ldots,n$.
\end{lemma}

To illustrate the use of Lemma~\ref{l:pattern}, consider an
example from the work of Schwarz~\cite{Sch-70}, attaining the corresponding bound.
It is a strongly connected graph consisting
of two cycles, of lengths $6$ and $4$, displayed in the left part of Figure~\ref{f:schwarz}.
The greatest transient of a row of the associated Boolean matrix is $T_4(A)=11$, which is equal to
Schwarz's bound of Theorem~\ref{t:allbounds} with $n=7$ and $d=2$.
Now, let~$B$ be a matrix with pattern~$A$ such that node $4$ is critical.
On one hand, Main Theorem~\ref{t:mainres} ensures that $T_4(B)\le 11$. On the other hand, by Lemma~\ref{l:pattern}, we have~$T_4(B)\ge T_4(A)=11$.
Thus Schwarz's bound is attained by~$T_4(B)$. In particular,
consider any nonnegative matrix $B$ where all entries of the bigger cycle are equal to $1$,
and the two remaining nonzero entries are less than or equal to $1$. The associated digraphs of $A$ and $B$ are
displayed on Figure~\ref{f:schwarz}. It can be checked by direct computation that $T_4(B)=11$.
More examples of this kind can be constructed using the work of Shao and Li~\cite{SL-87}.
Observe that not all the nodes of the graph on the right-hand side of Figure~\ref{f:schwarz} are critical,
but it does attain the greatest possible transient of critical rows, because
node~$4$ is critical.

\begin{figure}
\centering
\begin{tabular}{cc}

\begin{tikzpicture}[shorten >=1pt,->,scale=2,>=latex']

\tikzstyle{vertex}=[circle,draw,minimum size=17pt,inner sep=1pt,thick]

\foreach \name/\angle/\text in {1/0/1, 2/300/2,3/240/3, 4/180/4, 5/120/5, 6/60/6}
\node[vertex,xshift=0cm,yshift=0cm] (\name) at (\angle:1cm) {$\text$};

\draw node[vertex,xshift=0cm,yshift=0cm] (7)  {$7$};


\draw [->] (1) to node[right] {$1$}(2);

\draw [->] (2) to node[below] {$1$}(3);

\draw [->] (3) to node[left] {$1$} (4);
\draw [->] (4) to node[above] {$1$}(5);
\draw [->] (5) to node[above] {$1$}(6);
\draw [->] (6) to node[above] {$1$}(1);
\draw [->] (3) to node[above] {$1$}(7);
\draw [->] (7) to node[above] {$1$}(1);

\end{tikzpicture}
&
\begin{tikzpicture}[shorten >=1pt,->,scale=2,>=latex']

\tikzstyle{vertex}=[circle,draw,minimum size=17pt,inner sep=1pt,thick]

\foreach \name/\angle/\text in {1/0/1, 2/300/2,3/240/3, 4/180/4, 5/120/5, 6/60/6}
\node[vertex,xshift=0cm,yshift=0cm] (\name) at (\angle:1cm) {$\text$};

\draw node[vertex,xshift=0cm,yshift=0cm] (7)  {$7$};


\draw [->] (1) to node[right] {$1$}(2);

\draw [->] (2) to node[below] {$1$}(3);

\draw [->] (3) to node[left] {$1$} (4);
\draw [->] (4) to node[above] {$1$}(5);
\draw [->] (5) to node[above] {$1$}(6);
\draw [->] (6) to node[above] {$1$}(1);
\draw [->] (3) to node[left] {$0.5$}(7);
\draw [->] (7) to node[above] {$0.5$}(1);

\end{tikzpicture}
\end{tabular}
\caption{Schwarz's example (left) and its max-algebraic version
(right)\label{f:schwarz}}
\end{figure}
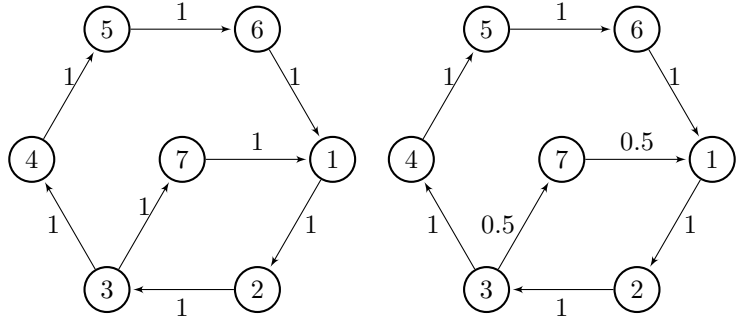

Another way to produce examples is to use the map
$\phi_1:\Rpnn\rightarrow\bool^{n\times n}$ that maps~$A$ to
$B=\phi_1(A)$ such that $b_{i,j}=1$ if and only if $a_{i,j}=1$ and $b_{i,j}=0$
otherwise. It is not a morphism on $\Rpnn$ but it is not difficult
to check (generalizing Lemma~\ref{l:CAk}) that it is a morphism on
the following semigroups of~$\Rpnn$, defined for each subset~$X$
of~$N$:
$$\mathcal{S}_{X}=\left\{ A\in\Rpnn\mid \lambda(A)=1\,,\, A\textnormal{ is strictly
visualized}\,,\, N_c(A)=X\right\}$$

For any matrix~$A$ in~$\mathcal{S}_{X}$ the critical edges have weight~$1$, so that $\crit(A)=\crit(\phi_1(A))$.
Obviously, if~$\lambda(A)=1$ and~$A$ is strictly visualized, $A\in \mathcal{S}_{N_c(A)}$, so that any such matrix satisfies
$T_k(A)\ge T_k\left(\phi_1(A)\right)$. It extends to general matrices in the following way.

For $A\in\Rpnn$, let $A^C$ be the {\em critical matrix\/} of $A$, with
entries
\begin{equation*}
a^C_{ij}=
\begin{cases}
1, &\text{if $(i,j)\in E_c(A)$},\\
0, &\text{otherwise}.
\end{cases}
\end{equation*}



\begin{lemma}[cf.~{\cite[Corollary 2.9]{BSST}}]
\label{l:bsst}
$T_k(A^C)\leq T_k(A)$ for all $k\in N_c(A)$.
\end{lemma}

Lemma~\ref{l:bsst} shows that if an unweighted digraph $\crit$ on $n$ nodes attains a given bound
for some $T_k$, then any
$n\times n$ matrix with entries in $[0,1]$, whose critical graph is $\crit$, attains it as well (for the same $k$).
In the following example, the first matrix attains Wielandt's bound
($T_5(A)=17$) second matrix attains Dulmage and Mendelsohn's bound  ($T_4(B)=14$), since their critical
graphs attain the corresponding Boolean bounds.

\begin{equation*}
A=
\begin{pmatrix}
0 & 1 & 0.1 & 0.2 & 0\\
0 & 0 & 1 & 0 & 0.3\\
0 & 0.4 & 0 & 1 & 0\\
1 & 0 & 0.5 & 0 & 1\\
1 & 0 & 0 & 0 & 0.9\\\end{pmatrix},\quad
B=
\begin{pmatrix}
0 & 1 & 0.4 & 0.5 & 0\\
0.1 & 0 & 1 & 0 & 0\\
1 & 0 & 0.2 & 1 & 0.7\\
0 & 0.9 & 0 & 0 & 1\\
1 & 0.2 & 0.6 & 0.7 & 0\\\end{pmatrix}.
\end{equation*}

\if{ Thanks to Lemma~\ref{l:bsst} all examples attaining the Boolean
bounds can be transformed into matrices with positive entries only,
attainin For example, it proves that there are matrices with
positive entries of any dimension reaching the bound of Wielandt.
}\fi


\if{
Finally, there is a third way to modify matrices that reaches
the bounds so that they still do:
\begin{lemma}\label{l:AddingEdges}
Let~$B$ be a visualized matrix with~$\lambda(B)=1$ and $k$ a
critical node. Let~$C$ be a matrix whose entries are strictly less
than $\min \{b^{(t)}_{k,j}>0\mid t\ge 1, j\in\{1,\cdots n\}\}$
and~$A=B\oplus C$

Then $T_k(A)\ge T_k(B)$.
\end{lemma}

\begin{proof}[To detail using optimal walk interpretation]
Since $B$ is visualized and~$\lambda(B)=1$, the entries of~$B$ are
at most~$1$, and so are the entries of~$C$ and finally of~$A$.

The key point is to realize that $C$ is built in such a way that for
any~$t$ and~$j$, if $b^{(t)}_{k,j}>0$, then
$a^{(t)}_{k,j}=b^{(t)}_{k,j}$.

Now, we take a period~$p$ of both~$a^{(t)}_{k,j}$
and~$b^{(t)}_{k,j}$ and consider the sequence $b^{(pt+t_0)}_{k,j}$.

This sequence is eventually constant. If the limit is zero, then the
limit of~$a^{(pt+t_0)}_{k,j}$ is strictly less than all the
$b^{(t)}_{k,j}$s, and it can be reached only when
$b^{(pt+t_0)}_{k,j}$ has reached~$0$, so the transient of
$a^{(pt+t_0)}_{k,j}$ is at least the transient of
$b^{(pt+t_0)}_{k,j}$.

If the limit of $b^{(pt+t_0)}_{k,j}$ is positive, then it is also
the limit of $a^{(pt+t_0)}_{k,j}$ and the same is true.

Since, this holds for any~$t_0$ and any~$j$, the Lemma is proved.
\end{proof}

}\fi


We conclude that the two lemmas provide us with some classes of
matrices attaining the bounds of Main Theorem~\ref{t:mainres}.
However, this characterization is far from being complete and leaves
vast possibilities of research.

\if{
\begin{equation}
\label{sch-ex}
A=
\begin{pmatrix}

\end
\end{equation}
}\fi

\if{
Let us also recall the following observation
on the critical part of max-algebraic powers in the periodic regime, in
order to complete the description.

\begin{proposition}
\label{p:multiple }
Let $k$ be an index of $N_c(A)$ and let $c$ be the cyclicity of the critical
component to which it belongs. If $t$ is a multiple of $c$, then
\begin{itemize}
\item[{\rm (i)}]  $A^{\otimes t}_{k\cdot}\leq (A^{\otimes c})^*_{k\cdot}$;
\item[{\rm (ii)}]  $A^{\otimes t}_{k\cdot}=(A^{\otimes c})^*_{k\cdot}$ when $t\geq T_k(A)$.
\end{itemize}
\end{proposition}
}\fi



\end{document}